\documentclass[a4paper,12pt,reqno,oneside]{amsart} 

\usepackage{setspace} 
\usepackage{typearea} 
\usepackage{amscd,amssymb,verbatim} 
\usepackage
{graphicx} 
\usepackage{enumitem}

\usepackage{xr-hyper}
\usepackage{cite}
\usepackage[colorlinks,backref,final,bookmarksnumbered,bookmarks]{hyperref}
\usepackage[backrefs]{amsrefs}

\usepackage{ifthen}
\newcommand{\arxiv}[2][]{\ifthenelse{\equal{#1}{}}
{\href{http://arxiv.org/abs/#2}{\tt arXiv:#2}}
{\href{http://arxiv.org/abs/math/#2}{\tt arXiv:math.#1/#2}}}

\theoremstyle{plain}
\newtheorem{maintheorem}{Theorem}
\newtheorem*{theorem*}{Theorem}
\newtheorem{theorem}{Theorem}[section]
\newtheorem{lemma}[theorem]{Lemma}
\newtheorem*{lemma*}{Lemma}
\newtheorem{corollary}[theorem]{Corollary}

\theoremstyle{definition}
\newtheorem{example}[theorem]{Example}

\newtheoremstyle{remark}
{}{}{}{}{\itshape}{}{ }{\thmname{#1}\thmnumber{ \itshape #2.}}
\theoremstyle{remark}
\newtheorem{remark}[theorem]{Remark}

\def\x{\times}
\def\but{\setminus}
\def\eps{\varepsilon}
\def\emptyset{\varnothing}
\renewcommand{\:}{\colon}
\def\N{\mathbb{N}}
\def\K{\mathcal{K}}
\def\R{\mathbb{R}}
\def\Z{\mathbb{Z}}
\def\Q{\mathbb{Q}}
\def\conpol{\bigtriangledown}

\DeclareMathOperator{\lk}{lk}

\begin{document}

\title{Topological isotopy and Cochran's derived invariants}
\author{Sergey A. Melikhov}
\address{Steklov Mathematical Institute of Russian Academy of Sciences,
ul.\ Gubkina 8, Moscow, 119991 Russia}
\email{melikhov@mi-ras.ru}

\begin{abstract} We construct a link in the $3$-space that is not isotopic to any PL link (non-ambiently).
In fact, there exist uncountably many $I$-equivalence classes of links.

The paper also includes some observations on Cochran's invariants $\beta_i$.
\end{abstract}

\maketitle

\section{Introduction} \label{intro}

An $m$-component {\it link} is an injective continuous map $S^1\x\{1,\dots,m\}\to\R^3$.
A {\it knot} is a one-component link.
A link is called {\it tame} if it is equivalent (i.e.\ ambient isotopic) to a PL link; otherwise it is 
called {\it wild}.
Two $m$-component (PL) links are called (PL) {\it isotopic} if they are (PL) homotopic through links.
It is not hard to see that PL isotopy as an equivalence relation on PL links is generated by ambient isotopy
and introduction of local knots.

In 1974, D. Rolfsen asked: ``If $L_0$ and $L_1$ are PL links connected by a topological isotopy, are they 
PL isotopic?'' \cite{Ro}.
This question is also implicit in Milnor's 1957 paper \cite{Mi}*{cf.\ Remark 2 and Theorem 10}.
It was shown by the author that the affirmative answer is implied by the following conjecture: PL isotopy 
classes of links are separated by finite type invariants that are well-defined up to PL isotopy \cite{Me}.
Here finite type invariants of links may be understood either in the usual sense (of Vassiliev and Goussarov) 
or in the weaker sense of Kirk and Livingston \cite{KL}.

In the same paper \cite{Ro} Rolfsen also asked the following question:
``All PL knots are isotopic to one another. Is this true of wild knots?
A knot in the $3$-space which is not isotopic to a tame knot would have to be so wild as to fail to pierce 
a disk at each of its points.%
\footnote{In the terminology of Bing \cite{Bi}, a knot $K\:S^1\to\R^3$ {\it pierces} a $2$-disk $D\subset\R^3$ 
if $K(S^1)\cap D$ is a single point in the interior of $D$, and $K$ links $\partial D$.}
The `Bing sling' [\cite{Bi}] is such a candidate.''
This problem has been mostly approached with geometric methods.
D. Gillman showed that a certain knot pierces no disk but lies on a disk (and hence is isotopic to the unknot); 
however, no subarc of the Bing sling lies on a disk \cite{Gi}.
Nevertheless, M. Brin constructed a knot that at each point is locally equivalent to the Bing sling, 
but is isotopic to a PL knot \cite{Br}.

\begin{maintheorem} \label{th1}
There exists a $2$-component link that is not isotopic to any PL link.
\end{maintheorem}

The proof is not too difficult, and there is a natural temptation to try to build on it so as to cover the case 
of knots, e.g.\ by introducing an auxiliary second component.%
\footnote{Admittedly, this is one reason why Theorem \ref{th1} had to wait so long after its proof 
was presented by the author at the 2005 conference ``Manifolds and their Mappings'' in Siegen, Germany.}
But this is not an easy business.
If there exists a non-trivial isotopy invariant of knots, it must somehow take into account both global 
(by Brin's construction) and local information, at each point of the knot (by Rolfsen's remark).
Moreover, such an invariant cannot be an invariant of $I$-equivalence,%
\footnote{Two links $S^1\x\{1,\dots,m\}\to\R^3$ are called {\it $I$-equivalent} if they cobound an injective 
continuous map $S^1\x I\x\{1,\dots,m\}\to\R^3\x I$.
Obviously, isotopy implies $I$-equivalence.}
because of C. Giffen's shift-spinning construction \cite{Gif}:

\begin{theorem*}[Giffen] Every knot is $I$-equivalent to the unknot.
\end{theorem*}

Giffen did not claim this in \cite{Gif}, as he explicitly worked only with PL knots; published 
accounts of his construction which explicitly treat wild knots can be found in \cite{Mats} and \cite{AG}.
To better understand Giffen's construction it may also help to consult its simplified version,
with a picture instead of the formula, in the books by Daverman \cite{Da}*{\S12} (where it is used to show 
that the double suspension of Mazur's homology $3$-sphere is $S^5$) and Hillman \cite{Hi}*{Theorem 1.9} 
(where it is used to show that $F$-isotopy implies $I$-equivalence), but beware that this simplified 
version does not work for wild knots.

\begin{maintheorem} \label{th2}
There exist uncountably many $I$-equivalence classes of $2$-component links.
\end{maintheorem}

Theorem \ref{th2} is proved in \S\ref{realization}.
Two uncountable families of pairwise non-$I$-equivalent wild links are shown in Figures \ref{w-infty} 
and \ref{m10-1c}.
Theorem \ref{th1} is a direct consequence of Theorem \ref{th2}, since there exist only countably many 
isotopy classes (and hence also $I$-equivalence classes) of PL links.
A constructive proof of Theorem \ref{th1}, producing an explicit link that is not isotopic (and actually not
$I$-equivalent) to any PL link, is discussed in \S\ref{rationality}.

It is well-known that Milnor's $\bar\mu$-invariants are well-defined for wild links and are invariants of
their isotopy \cite{Mi} as well as $I$-equivalence \cite{Sta}, \cite{Ca}, \cite{Co0}.
However, all $\bar\mu$-invariants that do not vanish identically for algebraic reasons are realizable by
PL links \cite{Co3}*{Theorem 7.2}, \cite{Orr} and so are all their finite combinations \cite{Kr}.
As for infinite combinations, it follows from Higman's lemma on well-quasi-orderings that they simply 
do not exist (see Theorem \ref{mu-invariants} below).

To prove Theorems \ref{th1} and \ref{th2}, we use a version of the Kojima--Yamasaki $\eta$-function, which was 
originally introduced as an invariant of $I$-equivalence of PL links with vanishing linking number
\cite{KY} (see also \cite{MR}*{\S2.3}).
In fact, Kojima and Yamasaki wrote in their introduction \cite{KY}: ``In the study of the $\eta$-function, 
we became aware of the impossibility to define it for wild links. 
The reason is essentially due to the fact that the knot module of some wild knot is not $\Lambda$-torsion.''
In the same paper \cite{KY} they also introduced another (closely related) invariant, called 
the $\lambda$-polynomial, which they succeeded to define for wild links, but showed to be non-invariant 
under isotopy.

A neat geometric reformulation of the $\eta$-function, useful for practical computations, was found by 
T. Cochran, who also extended it to higher-dimensional links \cite{Co}. 
Namely, Cochran showed that the $\eta$-function is equivalent by a change of variable to a rational power series 
$C_L(z)=\sum_i\beta_i(L) z^i$ with integer coefficients, which admit a simple description in terms of
iterated intersections of Seifert surfaces (see details in \S\ref{Cochran} below).

In \cite{Me}, the author extended each $\beta_i$ to a $\Q$-valued Vassiliev invariant $\beta_i^\mho$ 
of order $2k+1$ (of all two-component PL links, with possibly nonzero linking number) and proved that 
for each two-component link $L$ and every $i$ there exists an $\eps>0$ such that all PL links $L'$ that 
are $\eps$-close to $L$ (in the sup metric) have equal $\beta_i^\mho(L')$.
Thus, each $\beta_i^\mho$ further extends to wild links; moreover, the extended invariant, 
still denoted $\beta_i^\mho$, is invariant under isotopy \cite{Me}.
But of course the resulting extension ``by continuity'' $C_L^\mho$ of the power series $C_L$ need not be 
rational for wild $L$.
Basically, this is where Theorem \ref{th1} comes from.

While the rationality of $C_L$ for PL links $L$ is well-known \cite{Co}, \cite{Jin1}, \cite{Jin2}, \cite{GL}, 
we note that its meaning can be further clarified by a noteworthy formula.
Namely, $C_L$ turns out to be a quotient of two one-variable Conway polynomials:

\begin{maintheorem} \label{th3} Given a PL link $L=(K_0,K_1)$ with $\lk(L)=0$, let $\Lambda=(K_0\#_b K_1,-K_0^+)$, 
where $K_0\#_b K_1$ is any band connected sum and $-K_0^+$ is a zero pushoff of $K_0$ with reversed orientation,
disjoint from the band.
Then \[-zC_L(-z^2)=\frac{\nabla_\Lambda(z)}{\nabla_{K_1}(z)}.\]
\end{maintheorem}

As explained in \S\ref{rationality}, Theorem \ref{th3} follows from the Tsukamoto--Yasuhara factorization 
theorem \cite{TY}.
In \S\ref{rationality} we also discuss some other results in this spirit, including one which offers another 
``explanation'' of the rationality of $C_L$ for PL links (see Remark \ref{conpol}).
 
Theorem \ref{th2} has the following source:

\begin{maintheorem} \label{th4} $C_L^\mho$ is an $I$-equivalence invariant for wild links $L$ with $\lk(L)=0$.
\end{maintheorem}

By contrast, $C_L^\mho$ is not an $I$-equivalence invariant, nor even a concordance invariant in the case of 
PL links $L$ with $\lk(L)=1$ (see \cite{MR}*{\S2.2}). 

In order to prove Theorems \ref{th2} and \ref{th4} and to facilitate computation of $C_L^\mho$ for specific wild 
links, we observe in \S\ref{Cochran} below that a minor modification of Cochran's original construction actually 
applies directly to wild links, without using PL approximations.
(As a byproduct, the proofs of Theorems \ref{th1} and \ref{th2} do not really depend on \cite{Me}.)
The only real modification is to replace Seifert surfaces by $h$-Seifert surfaces.
(Wild links generally do not have Seifert surfaces, at least if the latter are understood as manifolds with
boundary.)
However, to make sure that everything goes through in the more general setting, we have to include the details,
taking into account that they are not always present in Cochran's paper \cite{Co}.%
\footnote{Apart from the explicitly omitted arguments, the assertion ``$H_{n+1}(E(W))$ is trivial'' in 
the proof of the ``cornerstone'' Theorem 4.2 in \cite{Co} is incorrect for $n=1$, so this proof needs 
to be repaired.
Some details appear in a slightly different context in a later paper by Cochran 
\cite{Co4}*{Proof of Proposition 2.5}.}
Indeed, we make use of Steenrod homology and \v Cech cohomology, or rather the Alexander duality which they 
satisfy.%
\footnote{See \cite{M2}*{\S4} for a down to earth exposition in a few pages, which explains, in particular, that 
the Jordan curve theorem is not something deep and mysterious, as many people think, but a trivial consequence 
of the Poincar\'e duality for PL manifolds along with Milnor's lemma on mapping telescopes.}

\section{Invariants} \label{Cochran}

Let us order the components of $S^1\sqcup\dots\sqcup S^1$ and fix orientations of $S^1$ and $S^3$.
Then every link $l\:S^1\sqcup\dots\sqcup S^1\to S^3$ (possibly wild) determines an {\it oriented ordered link}, 
that is, the subset $l(S^1\sqcup\dots\sqcup S^1)$ of $S^3$ endowed with a numbering of its components and with 
a choice of their orientations.
Conversely, since every orientation preserving homeomorphism of $S^1$ is isotopic to the identity, the isotopy 
class of a link is determined by the corresponding oriented ordered link.
A one-component oriented ordered link is called an {\it oriented knot}.
Breaking with the definitions of \S\ref{intro}, in what follows we will use ``link'' in the sense 
``oriented ordered link'' and ``knot'' in the sense ``oriented knot''.

An {\it $h$-Seifert surface} for a knot $K$ is an oriented properly embedded smooth surface in 
$S^3\but K$ whose class in $H_2^\infty(S^3\but K)\simeq H_2(S^3,K)$ corresponds to $[K]\in H_1(K)$ under 
the isomorphism $\partial_*\:H_2(S^3,K)\to H_1(K)$.
An $h$-Seifert surface need not be a true Seifert surface (even if $K$ is PL or smooth) because its closure 
in $S^3$ may fail to be a manifold with boundary.

Let $L=(K,K')$ be a two-component link $S^3$ with $\lk(L)=0$.

\begin{lemma} \label{Seifert}
$K$ has an $h$-Seifert surface that is disjoint from $K'$.
\end{lemma}

\begin{proof}
The inclusion map $H_1(K')\to H_1(S^3\but K)$ is zero, and consequently so is the restriction map 
$H^1(S^3\but K)\to H^1(K')$ (using the naturality of the $\smallfrown$-product).
Hence the image of $[K]$ under the Alexander duality $H_1(K)\to H^1(S^3\but K)$ is the image of 
a class $\xi_K\in H^1(S^3\but K,\,K')$, which is unique since the restriction map 
$H^0(S^3\but K)\to H^0(K')$ is an isomorphism. 
This $\xi_K$ may be identified with the homotopy class of a map of pairs
$f_K\:(S^3\but K,\,K')\to (S^1,\,\{1\})$.
We may assume that $f_K$ is a smooth map transverse to $-1\in S^1$.
Then $\Sigma:=f^{-1}(-1)$ is the desired $h$-Seifert surface.
\end{proof}

Now let $\Sigma$ be an $h$-Seifert surface for $K$ that is disjoint from $K'$ and $\Sigma'$ be 
an $h$-Seifert surface for $K'$ that is disjoint from $K$.
We may assume that $\Sigma$ and $\Sigma'$ meet transversely along a closed oriented $1$-manifold $F$.
Since $\Sigma$ and $\Sigma'$ are oriented, they are framed (i.e.\ their normal bundles in $S^3$ are endowed 
with natural framings), and hence so is $F$.

The {\it Sato--Levine invariant} $\beta(L)$ is the self-linking number of $F$, that is, the total linking 
number $\lk(F,F^{++})$, where $F^{++}$ is a pushoff of $F$ along the sum of the two vectors of the framing.
In other words, $\beta(L)$ is the framed bordism class of $F$, which by the Pontryagin construction
is identified with an element of $\pi_3(S^2)$.

Next, by attaching a finger to $\Sigma$ we may assume that $F$ is nonempty, and by attaching tubes to $\Sigma$ 
along paths in $\Sigma'$ we may assume that $F$ is connected.%
\footnote{An alternative approach is to define $\beta_i$ for possibly disconnected $F$.
The proof of Theorem \ref{main} applies without changes to these more general invariants.}
Then $\partial_1 L:=(F^{++},K')$ is a two-component link.
Since $F^{++}$ is disjoint from $\Sigma'$, we have $\lk(\partial_1 L)=0$.
Let $\beta_1(L)=\beta(L)$ and $\beta_{n+1}(L)=\beta_n(\partial_1 L)$.
In closed terms, $\beta_n(L)=\beta(\partial_1\dots\partial_1 L)$.

\begin{theorem} \label{main}
(a) $\beta$ is well-defined and is invariant under $I$-equivalence.

(b) Each $\beta_i$ is well-defined and is invariant under $I$-equivalence.
\end{theorem}

\begin{proof}[Proof. (a)] Suppose that $L_0=(K_0,K'_0)$ and $L_1=(K_1,K'_1)$ are links 
related by an $I$-equivalence $\Lambda=(\Delta,\Delta')$ in $S^3\x I$.
Let $\Sigma_i$, $\Sigma'_i$, $F_i$ and $F^{++}_i$ be as above for each $L_i$, where $i=0,1$.
By the Pontryagin construction we may assume that the $\Sigma_i$ arose from certain $\xi_{K_i}$ and $f_{K_i}$
as in the proof of Lemma \ref{Seifert}.
 
Let $\hat\Delta=CK_0\cup\Delta\cup CK_1\subset C(S^3\x 0)\cup S^3\x I\cup C(S^3\x 1)=S^4$.
Since the inclusion map $H_1(K_0')\to H_1(\Delta')$ is an isomorphism and $K_0'$ bounds $\Sigma_0'$
in $S^4\but\hat\Delta$, the inclusion map $H_1(\Delta')\to H_1(S^4\but\hat\Delta)$ is zero.
Consequently, so is the restriction map $H^1(S^4\but\hat\Delta)\to H^1(\Delta')$ (using the naturality of 
the $\smallfrown$-product).
Then the image of $[\hat\Delta]$ under the Alexander duality $H_2(\hat\Delta)\to H^1(S^4\but\hat\Delta)$ is 
the image of a $\xi_{\hat\Delta}\in H^1(S^4\but\hat\Delta,\,\Delta')$, which is unique since 
the restriction map $H^0(S^4\but\hat\Delta)\to H^0(\Delta')$ is an isomorphism. 
This $\xi_{\hat\Delta}$ may be identified with the homotopy class of a map
$f_{\hat\Delta}\:(S^4\but\hat\Delta,\,\Delta')\to (S^1,\,\{1\})$.
It is easy to see that the classes $\xi_{K_i}\in H^1(S^3\but K_i,\,K'_i)$, where $i=0,1$, are 
the restrictions of $\xi_{\hat\Delta}$.
Then we may assume that $f_{K_0}$ and $f_{K_1}$ are the restrictions of $f_{\hat\Delta}$ and that 
$f_{\hat\Delta}$ is a smooth map transverse to $-1\in S^1$.
Then $\hat W:=f_{\hat\Delta}^{-1}(-1)$ is an $h$-Seifert surface for $\hat\Delta$, that is, an oriented 
properly embedded smooth $3$-manifold in $S^4\but\hat\Delta$, whose class in 
$H_3^\infty(S^4\but\hat\Delta)\simeq H_3(S^4,\hat\Delta)$ corresponds to 
$[\hat\Delta]\in H_2(\hat\Delta)$ under the isomorphism $\partial_*\:H_3(S^4,\hat\Delta)\to H_2(\hat\Delta)$.
Moreover, $\hat W$ meets $S^3\x\partial I$ in $\Sigma_0\cup\Sigma_1$ and is disjoint from $\Delta'$.
Thus $\Delta$ {\it $h$-bounds} the bordism $W:=\hat W\cap S^3\x I$ between $\Sigma_0$ and $\Sigma_1$, 
which is disjoint from $\Delta'$.

Similarly, $\Delta'$ $h$-bounds a bordism $W'$ between $\Sigma'_0$ and $\Sigma'_1$, disjoint from $\Delta$.
We may assume that $W$ and $W'$ meet transversely along a surface $\Phi$ with boundary $F_0\cup -F_1$.
Since $W$ and $W'$ are oriented, they are framed, and hence so is $\Phi$.
Thus $\Phi$ is a framed bordism between $F_0$ and $F_1$.
So we obtain $\beta(L_0)=\beta(L_1)$.
\end{proof}

\begin{proof}[(b)] Let us observe that the proof of (a) still works if the hypothesis that $\Delta$ is 
homeomorphic to $S^1\x I$ extending the homeomorphism $K_0\cup K_1\to S^1\x\partial I$ is replaced by 
the following weaker hypothesis:
\begin{enumerate} 
\item $\Delta$ is an oriented surface with boundary $K_0\cup-K_1$, and
\item $\Delta'$ $h$-bounds a bordism $W'$ between $\Sigma_0'$ and $\Sigma_1'$, disjoint from $\Delta$.
\end{enumerate}
The hypothesis that
\begin{enumerate}
\item[(3)] $\Delta'$ is homeomorphic to $S^1\x I$ extending the homeomorphism $K'_0\cup K'_1\to S^1\x\partial I$
\end{enumerate}
is not meant to be modified.
Thus the proof of (a) works to show that $\beta$ is invariant not only under $I$-equivalence, but also under 
{\it weak $I$-equivalence}, which is defined by the conditions (1)--(3), with the $h$-Seifert surfaces 
$\Sigma_0'$ and $\Sigma_1'$ bound by the existential quantifier.
Specifically, a proof that $\beta$ is invariant under weak $I$-equivalence is obtained from the proof of (a)
by replacing ``$I$-equivalence'' with ``weak $I$-equivalence'' and ``Similarly'' with ``By hypothesis''.

To complete the proof of (b), it suffices to show that if $L_0$ is weakly $I$-equivalent to $L_1$, then
$\partial_1 L_0$ is weakly $I$-equivalent to $\partial_1 L_1$.
To prove this, we will build on the proof of (a), modified as stated (in addition, we must now assume that 
$F_0$ and $F_1$ are connected).
Let $\Phi^{++}$ be a pushoff of $\Phi$ along the sum of the two vectors of the framing.
Then $\Phi^{++}$ is disjoint from $W'$, and hence $(\Phi^{++},\Delta')$ is a weak $I$-equivalence between
$\partial_1L_0=(F_0^{++},K_0')$ and $\partial_1L_1=(F_1^{++},K_1')$.
\end{proof}

Next we note that the extensions of Cochran's derived invariants to wild links with vanishing linking number
defined above and in \cite{Me} coincide:

\begin{theorem} \label{stability}
For every $2$-component link $L$ with $\lk(L)=0$, each $\beta_i(L)=\beta_i^\mho(L)$.
\end{theorem}

Theorem \ref{th4} follows from Theorems \ref{main} and \ref{stability}.

\begin{proof} We need to show that each $\beta_i$, defined as above, satisfies the following property:
for each two-component link $L$ with $\lk(L)=0$ there exists an $\eps>0$ such that every PL link $\tilde L$ 
that is $\eps$-close to $L$ satisfies $\beta_i(\tilde L)=\beta_i(L)$.
In fact, we will prove this without assuming $\tilde L$ to be PL.

Let $L=(K,K')$, let $\Sigma'$ be as above and let $\Sigma_1,\dots,\Sigma_i$ and $F_1,\dots,F_i$ be 
the consecutive instances of $\Sigma$ and $F$ that are produced by the definition of $\beta_i$.
Let $N$ be a closed polyhedral neighborhood of $K$, disjoint from $\Sigma'\cup\Sigma_2\dots\cup\Sigma_i$
(for instance, the union of all simplexes of a sufficiently fine triangulation of $S^3$ that meet $K$).
Similarly, let $N'$ be a closed polyhedral neighborhood of $K'$, disjoint from $\Sigma_1\cup\dots\cup\Sigma_i$.
If $\eps>0$ is sufficiently small, $\tilde L=(\tilde K,\tilde K')$ lies in the interior of $N\cup N'$.
Then it is easy to see from the proof of Lemma \ref{Seifert} that the $h$-Seifert surfaces $\tilde\Sigma_1$ 
and $\tilde\Sigma'$ for $\tilde K$ and $\tilde K'$ in $S^3\but\tilde K'$ and $S^3\but\tilde K$ can be chosen 
so as to differ from $\Sigma_1$ and $\Sigma'$ only within $N\cup N'$.
Then clearly $\beta_i(\tilde L)=\beta_i(L)$.
\end{proof}

\begin{remark} \label{properties}
Let $L=(K,K')$ be a link in $S^3$, and let $\Sigma$, $\Sigma'$, $F$ and $F^{++}$ be as above.

(a) Clearly, $\beta(L)=\lk(-F,-F^{++})=\beta(L')$, where $L'=(K',K)$ is obtained by
formally exchanging the components of $L$.

(b) Clearly, $\beta(L)=\lk(F,F^{0+})$, where $F^{0+}$ is the pushoff of $F$ within $\Sigma$, along the framing 
of $\Sigma'$.
This often facilitates computation.

(c) Clearly, $\beta(L)=\lk(F,F^{+-})=\lk(-F,-F^{+-})=\beta(K,-K')$, where $F^{+-}$ is the pushoff of $F$ 
along the difference of the two vectors of the framing.

(d) Each $\beta_i(L)=\beta_i(-K,K')=\beta_i(K,-K')$ since this holds for $i=1$ and $(F^{++},K')$ 
is equivalent to $(F^{+-},K')$ and $(F^{-+},K')$.
Thus to compute $\beta_i(L)$ we may forget the orientations of the components of $L$ and orient the surfaces
randomly.

(e) Clearly, $\beta_n(K',K)=\beta_n'(K,K')$, where $\beta_n'(L)=\beta(\partial_2\dots\partial_2 L)$
using $\partial_2 L:=(K,F^{++})$.
However, the ``mixed derivatives'' are not well-defined in general.
\end{remark}

\begin{remark} \label{mu}
For PL links with vanishing linking number $\beta=\bar\mu(1122)$ \cite{Co}*{\S9}, \cite{Co2}*{\S4}, \cite{SB}
and more generally the residue class $\bar\beta_n$ of $\beta_n$ modulo $\gcd(\beta_1,\dots,\beta_{n-1})$
is the same as $\bar\mu(\underbrace{1\ldots1}_{2n}22)$ \cite{Co3}*{Theorem 6.10}, \cite{Ste}
(see also \cite{Mi}*{\S4}).
For wild links the equation $\bar\beta_n=\bar\mu(\underbrace{1\ldots1}_{2n}22)$ remains valid, since each 
$\bar\beta_i$ is stable under small perturbations (by Theorem \ref{stability}) and so is each 
$\bar\mu$-invariant (see \cite{Mi}).

Other Milnor's $\bar\mu$-invariants (or rather Massey products in the cohomology of $S^3\but L$, which are 
well-known to contain the same information as $\bar\mu$-invariants) admit geometric descriptions in the spirit 
of that of the $\beta_i$ \cite{Co3}*{Proposition 6.5}. 
Here are three simplest examples of such constructions; as presented here, they apply to wild links.

\begin{itemize}
\item Let $L=(K,K')$ be such that $\lk(L)=\beta(L)=0$.
Let $\Sigma$, $\Sigma'$, $F$ and $F^{++}$ be as above.
Then $\lk(F^{++},F)=0$ due to $\beta(L)=0$ and $\lk(F^{++},K)=\lk(F^{++},K')=0$ since $F^{++}$ is disjoint
from both $\Sigma$ and $\Sigma'$.
Hence $F^{++}$ bounds a Seifert surface $\Sigma''$ disjoint from $K\cup F\cup K'$.
Let $E=\Sigma''\cap\Sigma$ and $E'=\Sigma''\cap\Sigma'$ (transverse intersections), and let $\delta(L)=\lk(E,E')$.
For PL links with true Seifert surfaces $\delta=\bar\mu(111222)$ \cite{Co3}*{Example 6.7}.
For wild links, the argument of Theorem \ref{stability} above shows that $\delta$ is stable under small 
perturbations, and hence the equation $\delta=\bar\mu(111222)$ remains valid.
In particular, $\delta$ is well-defined (under the hypothesis $\lk(L)=\beta(L)=0$) and is an $I$-equivalence 
invariant \cite{Sta}, \cite{Ca}, \cite{Co0}.

\item Let $L=(K,K',K'')$ have trivial pairwise linking numbers.
Then $K$ bounds an $h$-Seifert surface in $S^3\but (K'\cup K'')$ and $K'$ bounds an $h$-Seifert surface
in $S^3\but (K\cup K'')$. 
Let $F=\Sigma\cap\Sigma'$ (transverse intersection), and let $\beta(L)=\lk(F,K'')$.
If $\Sigma''$ is an $h$-Seifert surface for $K''$ in $S^3\but (K\cup K')$, clearly $\beta(L)$ equals 
the algebraic number of transverse intersection points $\#\Sigma\cap\Sigma'\cap\Sigma''$.%
\footnote{See \cite{M3} for a discussion of related geometric constructions of $\mu(123)$ by P. Akhmetiev 
(based on neatly chosen null-homotopies in place of Seifert surfaces) and U. Koschorke (rediscovered 
by Melvin--Morton; based on arbitrary null-homotopies, but with correction terms) and their algebraic
significance.}
For PL links with true Seifert surfaces it is well-known that $\beta=\mu(123)$ \cite{Co2}; wild links
can be treated as above.

\item Let $L=(K,K',K'')$ have trivial pairwise linking numbers and $\mu(123)$.
Then there exist $\Sigma$, $\Sigma'$ and $\Sigma''$ as above with $\Sigma\cap\Sigma'\cap\Sigma''=\emptyset$.
(First make $\Sigma\cap\Sigma'$ connected, as explained in \S\ref{Cochran}.
Then some pair of triple points with opposite signs are connected by an arc in $\Sigma\cap\Sigma'$ containing 
no other triple points.
So this pair can be canceled by attaching a tube to $\Sigma''$ along this arc.
This process can be repeated until all triple points are eliminated.)
Let $F=\Sigma\cap\Sigma'$ and $F'=\Sigma\cap\Sigma''$ (transverse intersection), and let $\gamma(L)=\lk(F,F')$.
For PL links with true Seifert surfaces $\gamma=-\mu(1123)$ according to \cite{Co3}*{Appendix B}; wild links
can be treated as above.
\end{itemize}

However, apart from the $\beta_i$ (and the $\beta'_i$, see Remark \ref{properties}(e)), no other infinite
series of simulateously defined integer invariants seems to have been constructed in this fashion, although 
there are some candidates \cite{Li1}*{Remark 5.3} (see also \cite{TY}).
\end{remark}

\section{Realization} \label{realization}

\begin{figure}[h]
\includegraphics{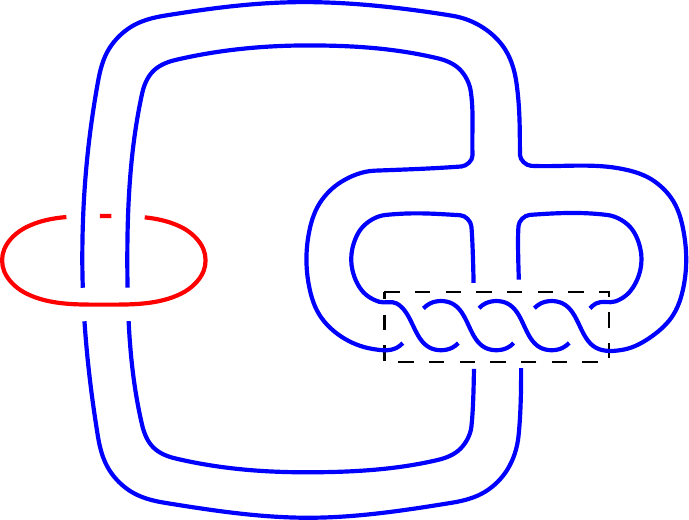}
\caption{The link $W_n$ for $n=2$.}
\label{wh2a}
\end{figure}

\begin{example} \label{Wh-example}
Let us compute $\beta(W_n)$, where the link $W_2$ is shown in Figure \ref{wh2a} and $W_n$ is obtained 
by replacing the two positive full twists in the dashed rectangle with $n$ positive full twists ($n\in\Z$).
In fact, it is easy to see that $W_0$ is the unlink and $W_1$, $W_{-1}$ are the left-handed and right-handed 
versions of the Whitehead link.

\begin{figure}[h]
\includegraphics{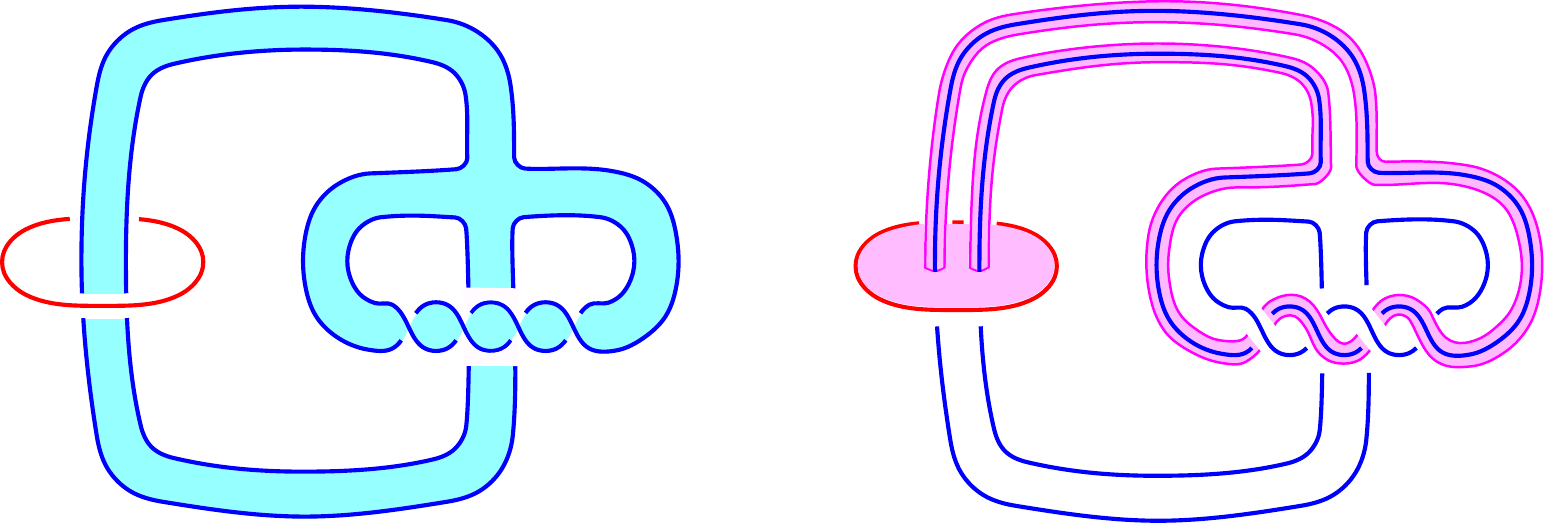}
\caption{$\Sigma$ and $\Sigma'$.}
\label{wh2b}
\end{figure}

Figure \ref{wh2b} shows a choice of the $h$-Seifert surfaces $\Sigma$ and $\Sigma'$ (in this case, they are
Seifert surfaces).
Figure \ref{wh2c} shows their intersection $F$ and its outward pushoff $F^{0+}$ on $\Sigma$ along a framing
of $\Sigma'$.
We have $\beta(W_2)=\lk(F,F^{0+})=2$ and similarly $\beta(W_n)=n$.
\end{example}

\begin{figure}[h]
\includegraphics{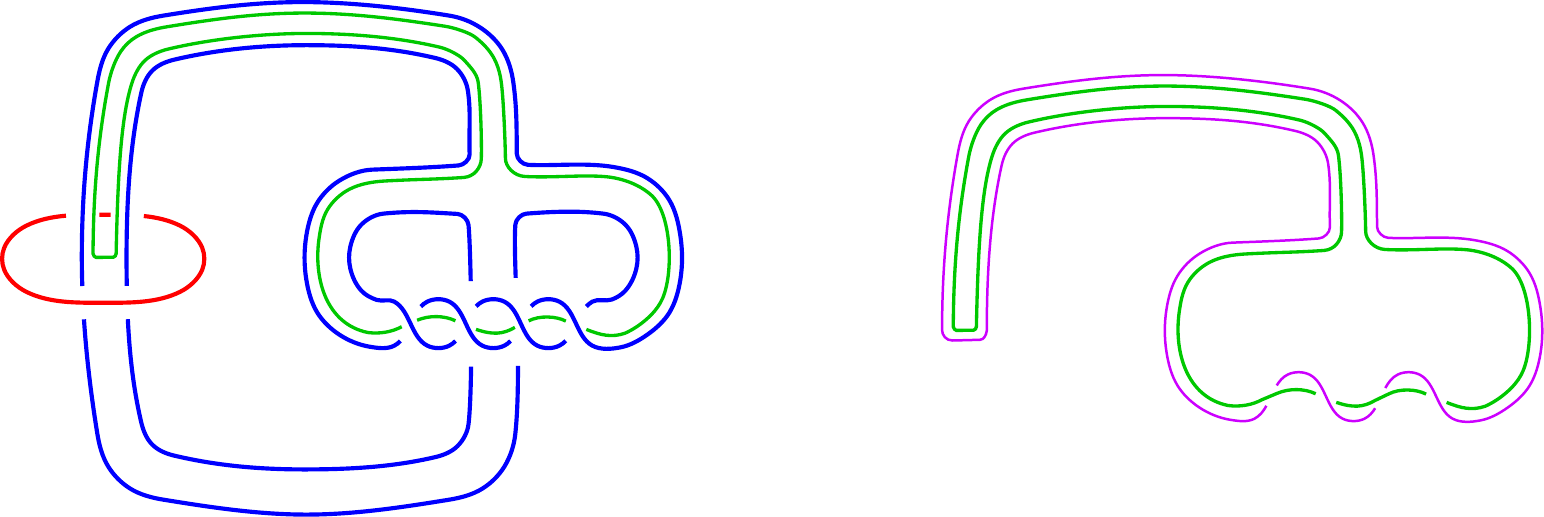}
\caption{$(W_2,F)$ and $(F,F^{0+})$.}
\label{wh2c}
\end{figure}

\begin{figure}[h]
\includegraphics{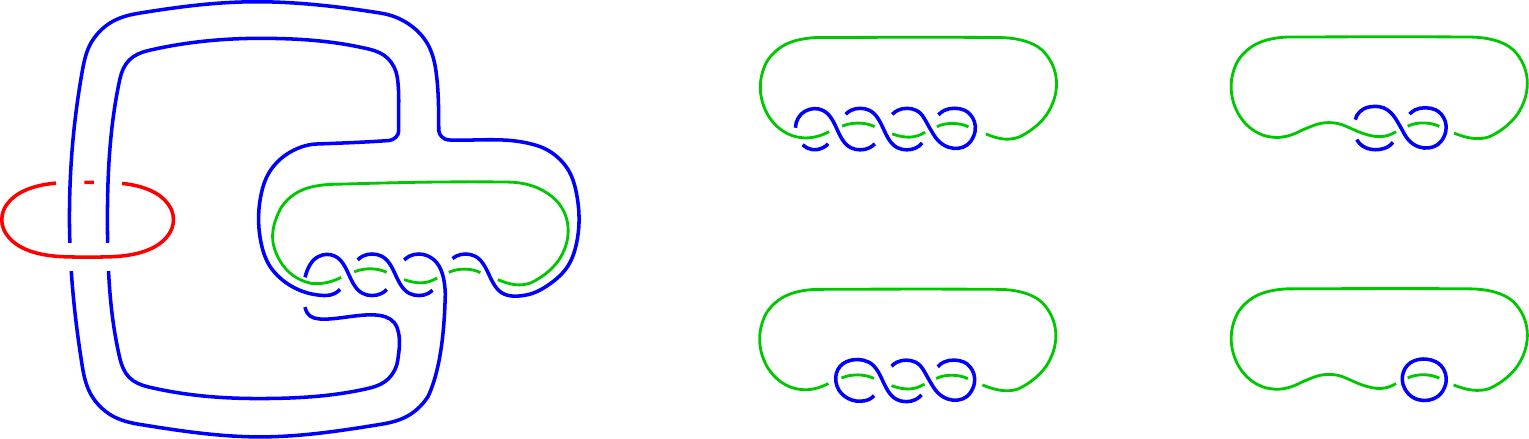}
\caption{$(W_2,F)$ and $\partial_1(W_2)=(F,K')$.}
\label{wh2d}
\end{figure}

\begin{example} \label{Wh-example2}
Let us compute $\beta_i(W_n)$ and $\beta_i(W_n')$, where $W_n'=(K',K)$ is the formal transposition of 
$W_n=(K,K')$.
Obviously, $(K,F)$ is a trivial link, and it is clear from Figure \ref{wh2d} that $(F,K')$
is also a trivial link.
Thus $\beta_i(W_n)=\beta_i(W_n')=0$ for $n>1$.
\end{example}

\begin{figure}[h]
\includegraphics{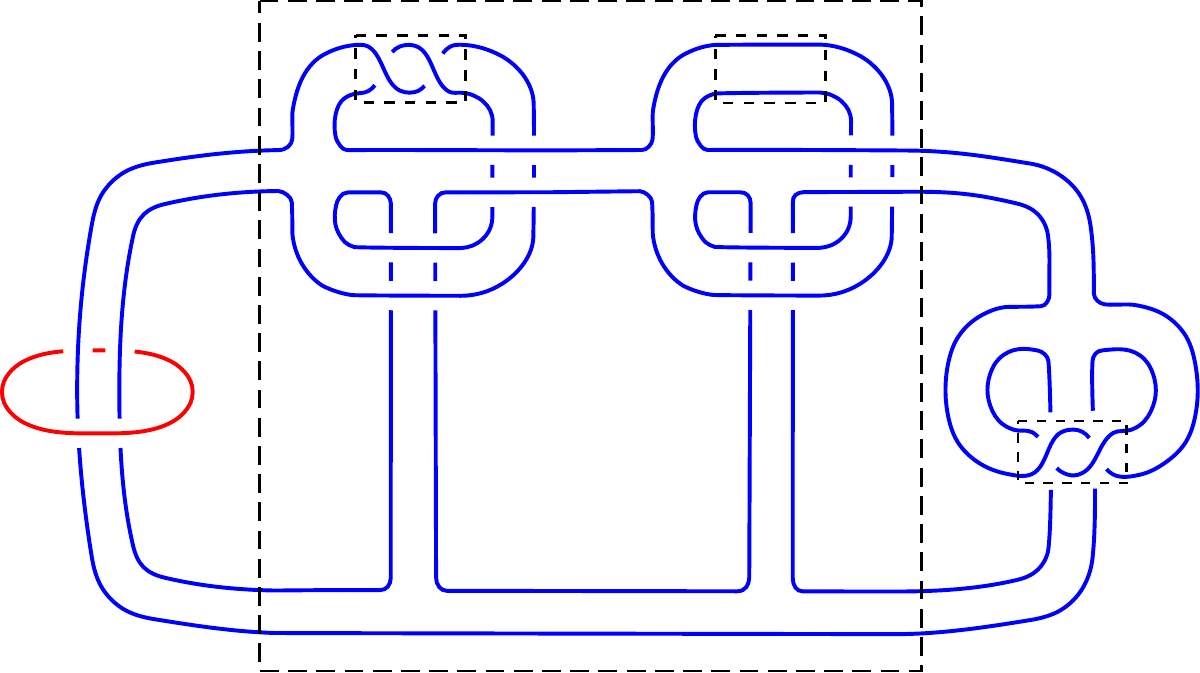}
\caption{The link $W_{n_1,\dots,n_m}$ for $m=3$ and $(n_1,n_2,n_3)=(1,0,-1)$.}
\label{w10-1a}
\end{figure}

\begin{example} \label{Milnor-example}
Let us compute $\beta_i(W_{n_1,\dots,n_m})$, where the link $W_{1,0,-1}$ is shown in Figure \ref{w10-1a} and 
$W_{n_1,\dots,n_m}$ is obtained by replacing the negative full twist in the rightmost dashed rectangle 
with $n_m$ positive full twists, and the two similar pictures in the large dashed rectangle with $m-1$ 
such pictures, with $n_1,\dots,n_{m-1}$ full twists in the small rectangles.

Thus $W_n$ is same as before.
The link $W_{0,\dots,0,1}$ is equivalent to Li's link \cite{Li}*{Figure 1}, which in turn is a twisted version 
of Milnor's link \cite{Mi}*{Figure 1}. 
The link $W_{0,\dots,0,-1}$ is equivalent to the mirror image of a link of Kojima and Yamasaki 
\cite{KY}*{Figure 9}.

\begin{figure}[h]
\includegraphics{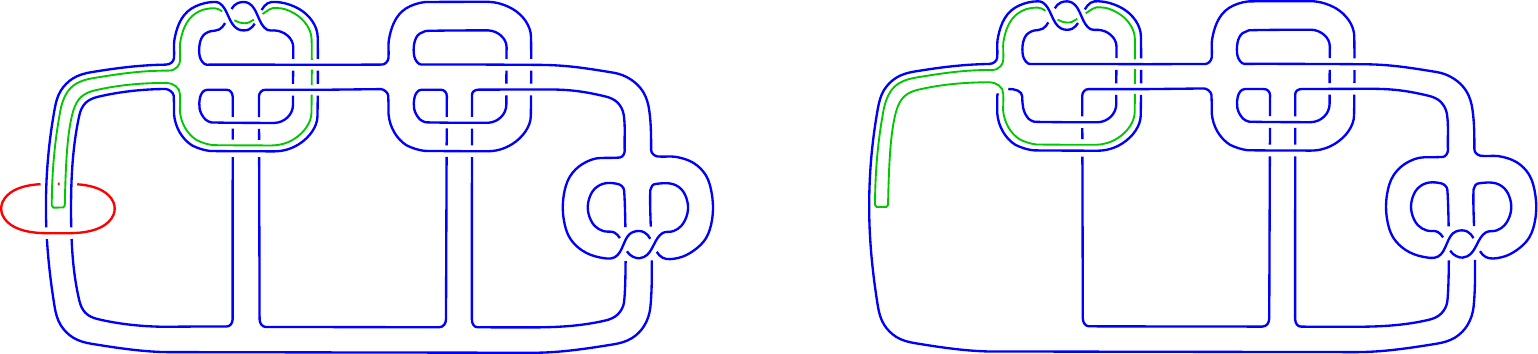}
\caption{$(K,F,K')$ and $\partial_1(W_{1,0,-1})=(F,K')$.}
\label{w10-1b}
\end{figure}

By using Seifert surfaces similar to those in Figure \ref{wh2b}, we obtain their intersection $F$ as shown
in Figure \ref{w10-1b} (left).
Clearly, $\beta(W_{1,0,-1})=\lk(F,F^{0+})=1$, and similarly $\beta_1(W_{n_1,\dots,n_m})=n_1$.

On the other hand, it is clear from Figure \ref{w10-1b} (right) and Figure \ref{w10-1c}
that $(F,K')$ is nothing but $W_{1,-1}$.
Similarly, $\partial_1(W_{n_1,\dots,n_m})=W_{n_2,\dots,n_m}$.

\begin{figure}[h]
\includegraphics{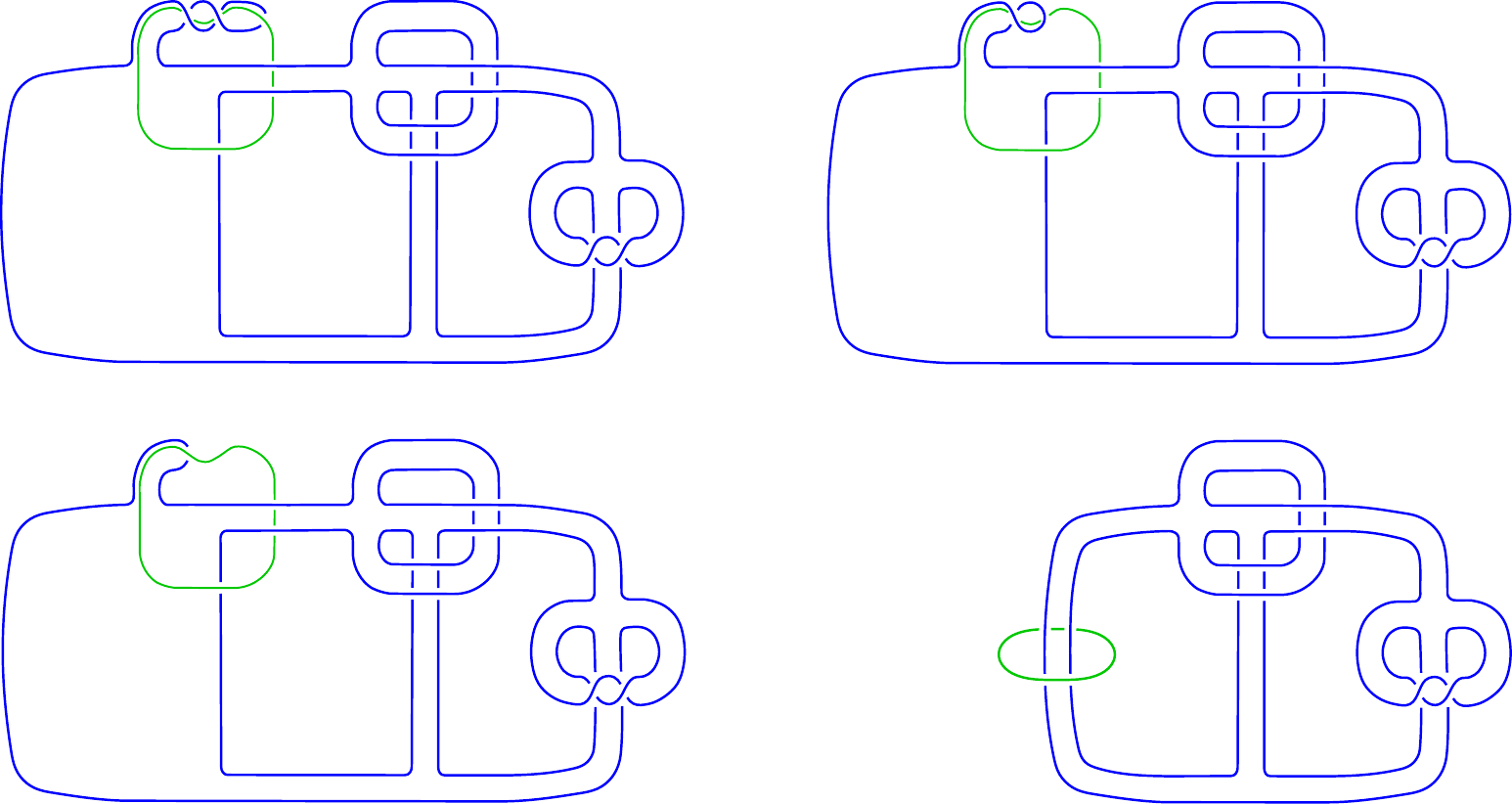}
\caption{$\partial_1(W_{1,0,-1})$ is equivalent to $W_{0,-1}$.}
\label{w10-1c}
\end{figure}

Hence $\beta_i(W_{n_1,\dots,n_m})=\beta_{i-1}(W_{n_2,\dots,n_m})$ as long as $i>1$ and $m>1$.
On the other hand, by the previous examples $\beta_1(W_{n_m})=n_m$ and $\beta_i(W_{n_m})=0$ for $i>1$.
It follows that 
\[\beta_i(W_{n_1,\dots,n_m})=
\begin{cases}
n_i\text{ for }i\le m,\\
0\text{ for }i>m.
\end{cases}\]

Finally, if $W'_{n_1,\dots,n_m}=(K',K)$ is the formal transposition of $W_{n_1,\dots,n_m}=(K,K')$, 
then each $\beta_i(W'_{n_1,\dots,n_m})=0$ unless $i=m=1$, using that $(K,F)$ is a trivial link.
\end{example}

\begin{figure}[h]
\includegraphics{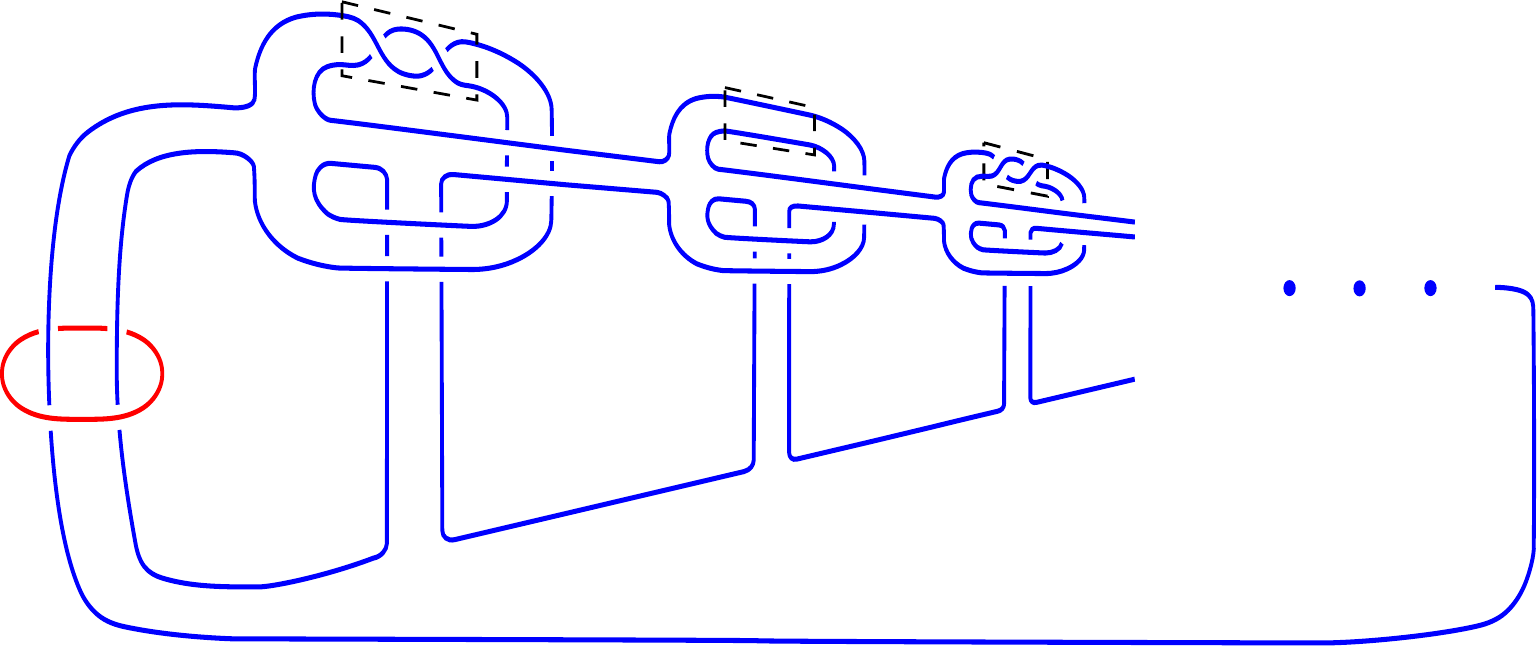}
\caption{The link $W_{n_1,n_2,n_3,\dots}$ where $(n_1,n_2,n_3)=(1,0,-1)$.}
\label{w-infty}
\end{figure}

\begin{example} \label{w-example}
For each sequence of integers $n_1,n_2,n_3,\dots$ let $W_{n_1,n_2,n_3,\dots}$ be the link 
shown in Figure \ref{w-infty} (with $n_i$ positive full twists in the $i$th dashed window).
Similarly to the previous example,
\[\beta_i(W_{n_1,n_2,n_3,\dots})=n_i.\]
\end{example}

Writing $C_L(x)=\sum_{i=1}^\infty\beta_i(L)x^i$, we have proved

\begin{theorem} \label{surjectivity}
For every power series $P\in\Z[[x]]$ there exists a link $L$ such that $C_L=P$.
\end{theorem}

Theorems \ref{main} and \ref{surjectivity} imply Theorem \ref{th2}.
In fact, the links $(W_{n_1,n_2,n_3,\dots})$ with each $n_i\in\{0,1\}$ already form an uncountable collection
of pairwise non-$I$-equivalent links.

\begin{remark}
It is easy see that for each finite sequence of integers $n_1,\dots,n_m$, the link $W_{n_1,\dots,n_m,0,0,\dots}$
is isotopic to $W_{n_1,\dots,n_m}$ (this is similar to \cite{Mi}*{Figures 3, 4, 5}).

The link $W_{0,0,\dots}$, which is isotopic to the trivial link, is the mirror image of a link of
Kojima and Yamasaki \cite{KY}*{Figure 7}.
If $W_{0,0,\dots}=(K,K')$, it is not hard to check that $\pi_1(S^3\but K')$ is isomorphic to
$\left<x_1,x_2,\ldots\mid x_ix_{i-1}x_i^{-1}=x_{i+1}x_ix_{i+1}^{-1}\right>$ (compare \cite{Fox}, 
\cite{Mi}*{p.\ 304}) and the isomorphism sends $[K]$ to $x_1^{-1}x_2x_1x_2$.
This group surjects onto $S_3$ via $x_{2i}\mapsto (12)$ and $x_{2i+1}\mapsto (13)$ (in fact, $K'$ is
tricolorable).
In particular, $K$ is not null-homotopic in $S^3\but K'$.
\end{remark}

\begin{figure}[h]
\includegraphics{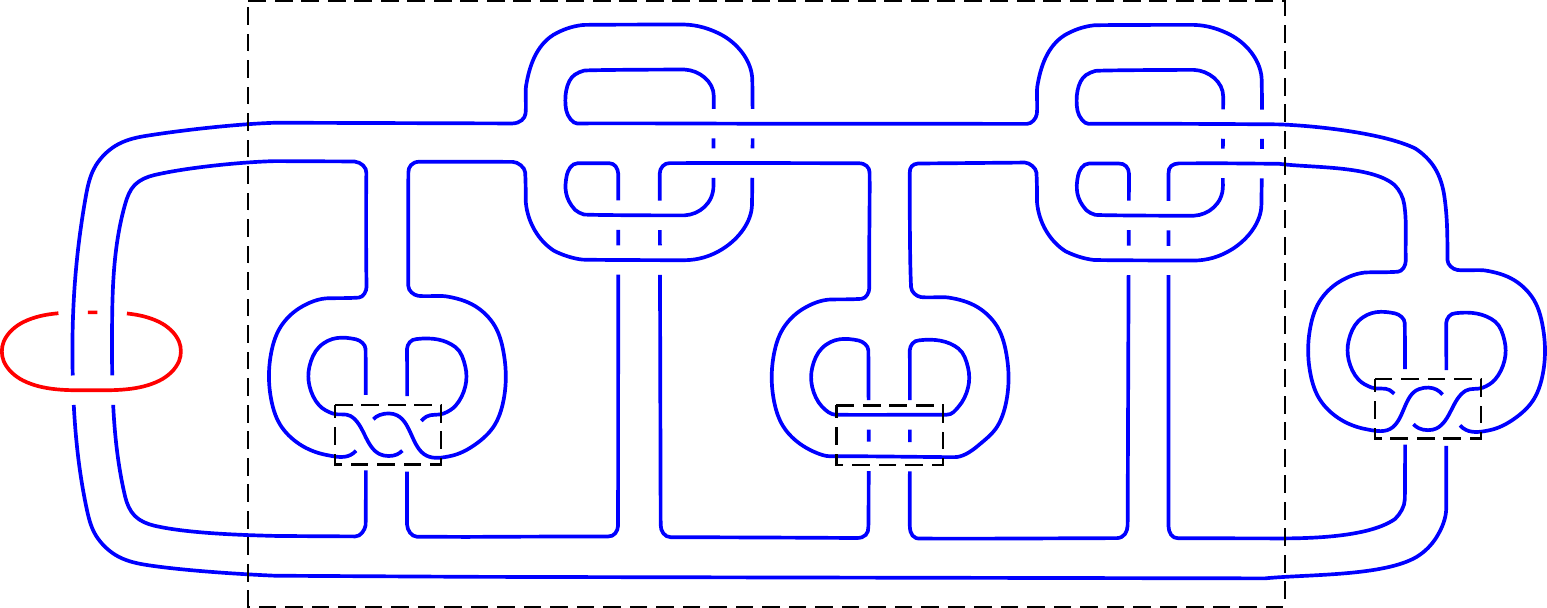}
\caption{The link $M_{n_1,\dots,n_m}$ for $m=3$ and $(n_1,n_2,n_3)=(1,0,-1)$.}
\label{m10-1a}
\end{figure}

\begin{example} \label{m-example}
Let us mention alternative links $M_{n_1,\dots,n_m}$ such that
\[\beta_i(M_{n_1,\dots,n_m})=
\begin{cases}
n_i\text{ for }i\le m,\\
0\text{ for }i>m.
\end{cases}\]
The link $M_{1,0,-1}$ is shown in Figure \ref{m10-1a} and again in Figure \ref{m10-1b}.
\end{example}

\begin{figure}[h]
\includegraphics{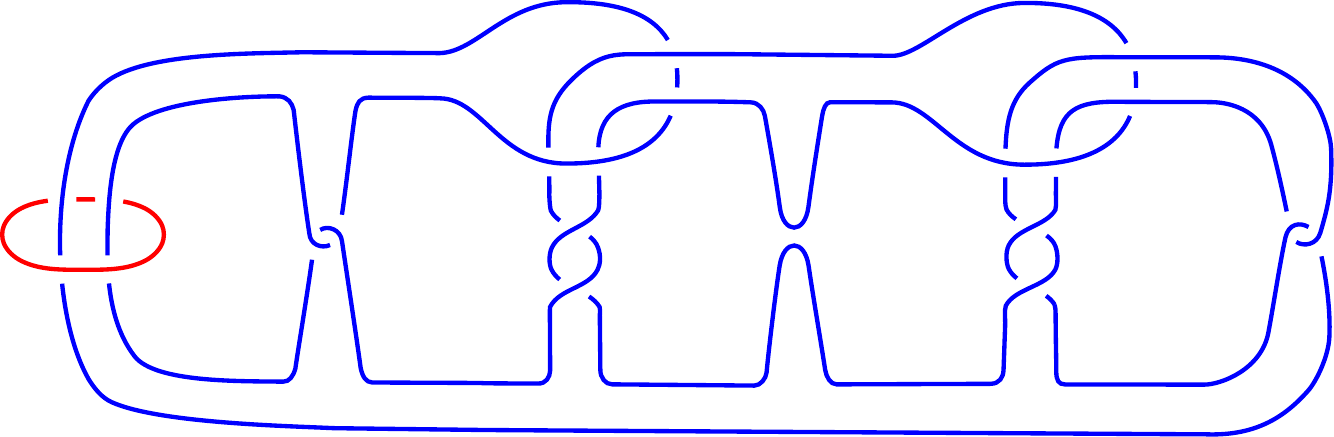}
\caption{$M_{1,0,-1}$ redrawn in a different way.}
\label{m10-1b}
\end{figure}

Figure \ref{m10-1c} shows alternative links $M_{n_1,n_2,\dots}$ for each sequence $n_1,n_2,\ldots\in\Z$
such that
\[\beta_i(M_{n_1,n_2,\dots})=n_i.\]
Let us note that $M_{0,0,\dots}$ is equivalent to $W_{0,0,\dots}$.
\begin{figure}[h]
\includegraphics{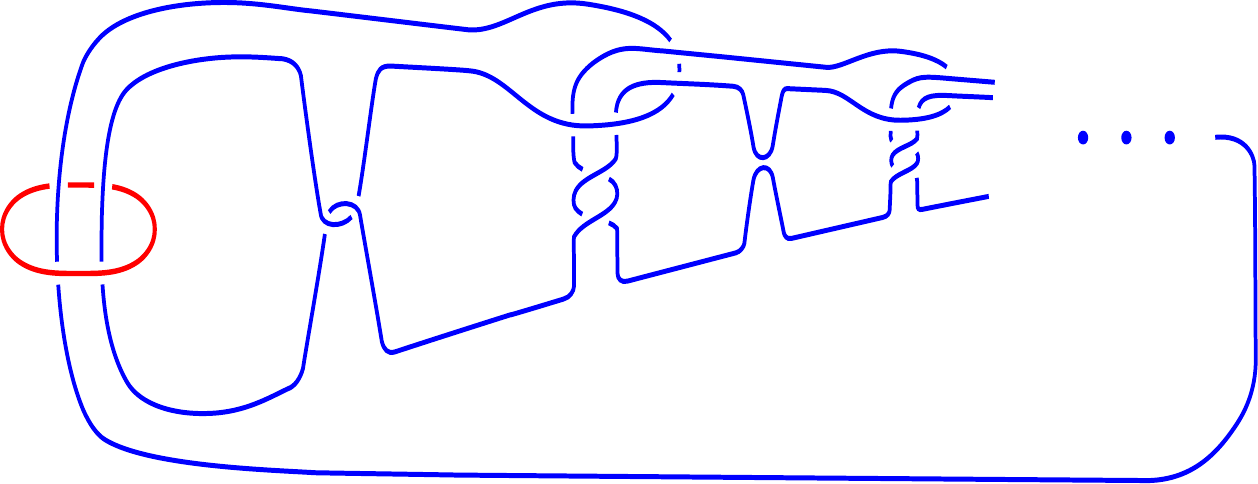}
\caption{The link $W_{n_1,n_2,\dots}$ where $(n_1,n_2)=(1,0)$.}
\label{m10-1c}
\end{figure}

\begin{example} The residue class $\bar\beta_n$ of $\beta_n$ modulo the ideal $(\beta_1,\dots,\beta_{n-1})$
of $\Z$ is the same as $\bar\mu(\underbrace{1\ldots1}_{2n}22)$ (see Remark \ref{mu}).
For the purposes of Theorems \ref{th1} and \ref{th2} the difference between $\beta_i$ and $\bar\beta_i$
is essential, because the sequence $\bar\beta_1(L),\bar\beta_2(L),\dots$ is eventually zero for every link $L$ 
(because $\Z$ contains no infinite ascending chain of ideals). 
\end{example}

This observation can be generalized to all $\bar\mu$-invariants using the following lemma.

\begin{lemma*}[Higman \cite{Hig}*{Theorem 2.1(i)$\Rightarrow$(iv) and Theorem 4.1}]
Every infinite sequence $I_1,I_2,\dots$ of multi-indices with entries from $\{1,\dots,m\}$ has 
an infinite subsequence $I_1',I_2',\dots$ such that each $I_k'$ embeds in $I_{k+1}'$.
\end{lemma*}

By a {\it multi-index} with entries from a set $S$ we mean any finite sequence of elements of $S$.
(These are alternatively known as words in the alphabet $S$.)
A multi-index $I$ {\it embeds} in a multi-index $J$ if $I$ is a subsequence of $J$ (in other words,
$J$ is obtained from $I$ by omitting some entries, but keeping the order of the remaining ones).

For the reader's convenience we include a direct proof of the case $m=2$.

\begin{proof}[Proof for $m=2$]
Given a multi-index $I$ (with entries from $\{1,2\}$), let $\lambda(I)$ be the length of $I$ and let $\alpha(I)$ be 
the number of decreases in $I$, that is, the number of $2$'s that are immediately followed by a $1$.
We consider two cases.

First suppose that $\alpha(I_n)$ is unbounded.
Given an $i$, let $j$ be such that $\alpha(I_j)\ge 2\lambda(I_i)$.
Then it is easy to see that $I_i$ embeds in $I_j$.

So we may assume that $\alpha(I_n)$ is bounded.
Then upon passing to a subsequence, we may assume that $\alpha(I_n)=\alpha$ does not depend on $n$.
Thus each 
\[I_n=(\underbrace{1\ldots1}_{\lambda_1(n)}\underbrace{2\ldots2}_{\lambda_2(n)}\ldots\ldots
\underbrace{1\ldots1}_{\lambda_{2\alpha-1}(n)}\underbrace{2\ldots2}_{\lambda_{2\alpha}(n)})\]
for some positive integers $\lambda_i(n)$, except for $\lambda_1(n)$ and $\lambda_{2\alpha}(n)$, which 
are nonnegative.
Upon passing to a subsequence we may assume that $\lambda_1(1)\le\lambda_1(2)\le\lambda_1(3)\le\dots$.
Given that, upon passing to a further subsequence we may additionally assume that 
$\lambda_2(1)\le\lambda_2(2)\le\lambda_2(3)\le\dots$.
Proceeding in this fashion, we will eventually get that $\lambda_k(1)\le\lambda_k(2)\le\lambda_k(3)\le\dots$
for all $k=1,\dots,2\alpha$.
\end{proof}

\begin{corollary} \label{mu-invariants} Let $L$ be an $m$-component link.
Then there exist only finitely many multi-indices $I$ (with entries from $\{1,\dots,m\}$) 
such that $\bar\mu_I(L)\ne 0$.
\end{corollary}

We recall that each $\bar\mu_I(L)$ is the residue class of a certain integer $\mu_I$ modulo the ideal 
$(\Delta_I)$ of $\Z$ generated by the set $\Delta_I:=\{\mu_J\mid J\text{ properly embeds in } I\}$.%
\footnote{The original definition of $\bar\mu_I(L)$ in \cite{Mi} also involves cyclic permutations, but they 
are redundant due to the cyclic symmetry of the $\bar\mu$-invariants \cite{Mi}*{Theorem 6}.}
This is the only property of $\bar\mu$-invariants that we need.

\begin{proof}
Suppose that there exists an infinite sequence of pairwise distinct multi-indices $I_1,I_2,\dots$ such that 
each $\bar\mu_{I_k}(L)\ne 0$.
By Higman's theorem we may assume that each $I_k$ embeds in $I_{k+1}$.
Then all $J$ that properly embed in $I_k$, as well as $I_k$ itself, all properly embed in $I_{k+1}$.
Therefore $\Delta_{I_k}\cup\{\mu_{I_k}\}\subset\Delta_{I_{k+1}}$.
On the other hand, $\mu_{I_k}\notin(\Delta_{I_k})$ by our hypothesis.
Hence $(\Delta_{I_k})\subsetneqq (\Delta_{I_k}\cup\{\mu_{I_k}\})\subset (\Delta_{I_{k+1}})$.
Thus we get an infinite chain of strictly ascending ideals of $\Z$, which is a contradiction.
\end{proof}

\section{Rationality} \label{rationality}

It is well-known (see e.g.\ \cite{St}*{4.1.1}) that a power series $R(x)=\sum_{i=1}^\infty n_ix^i\in\Z[[x]]$ 
is rational (i.e.\ is of the form $P/(1+xQ)$ for some polynomials $P,Q\in\Z[x]$) if and only if its coefficients 
satisfy a linear recurrence relation (i.e.\ there exist a $k_0\in\N$ and a finite sequence $c_1,\dots,c_m\in\Z$ 
such that $n_k=c_1n_{k-1}+\dots+c_mn_{k-m}$ for all $k>k_0$).

This implies, for example, that $C_L(z)=\beta_1z+\beta_2z^2+\beta_3z^3+\dots$ is rational if and only if 
$-zC_L(-z^2)=\beta_1z^3-\beta_2z^5+\beta_3z^7-\dots$ is rational.

\begin{theorem} \label{formula}
Let $L=(K_0,K_1)$ be a PL link, let $K$ be a band connected sum%
\footnote{That is, $b\:I\x I\to S^3$ is a PL embedding such that $b^{-1}(K_i)=I\x\{i\}$ and $K$ is 
obtained from $K_0\cup K_1$ by replacing $b(I\x\partial I)$ with $b(\partial I\x I)$.}
$K_0\#_b K_1$ and let $J_i$ be a pushoff of $K_i$, disjoint from the band, such that $\lk(J_i,K_0\cup K_1)=0$.

(a) $\nabla_L(z)=z\nabla_K(z)\big(\lk(L)-C_{\Lambda}(-z^2)\big)$, where $\Lambda=(J_1,K)$.

(b) If $\lk(L)=0$, then $-zC_L(-z^2)=\nabla_\Lambda(z)/\nabla_{K_1}(z)$, where $\Lambda=(K,-J_0)$.
\end{theorem}

Part (b) is identical with Theorem \ref{th3} and follows from (a).
Part (a) is a consequence of the two-component case of the Tsukamoto--Yasuhara factorization theorem \cite{TY}.

\begin{proof}[Proof. (a)] Let $F_0$ be a Seifert surface for $L$, disjoint from $J_1$.
We may assume that the core $C:=b(\{\frac12\}\x I)$ of the band $b$ meets $F_0$ transversely, and by rotating 
its end $b(\{\frac12\}\x [0,\eps])$ around $K_0$ we may assume that the algebraic number $(C\cdot F_0)$ of 
transverse intersections between $C$ and $F_0$ is zero.
Then by attaching pipes to $F_0$ going along (possibly overlapping) arcs in the interior of $C$ between 
algebraically canceling pairs in $C\cap F_0$ we may assume that $F_0$ is disjoint from $C$.
Then $F_0$ can be made disjoint also from the band $b(I\x I)$ by ambient isotopy.
Now by attaching a pipe to $F_0$ along the entire $C$ me may assume that $b(I\x I)\subset F_0$.
This can be made so that $J_1$ remains a close pushoff of $K_1$, disjoint from $F_0$.
Let $J_1^+$ be another close pushoff of $K_1$, disjoint from $F_0$ and $J_1$.

Let $F$ be the Seifert surface for $K$ obtained from $F_0$ by removing $b\big((0,1)\x I\big)$.
Let $X$ be the infinite cyclic covering of $S^3\but K$ and let us pick a fundamental domain $D$ of the action 
of $\Z$ on $X$ whose boundary consists of two consecutive lifts of $F$.
Let $\tilde J_1$ and $\tilde J_1^+$ be the lifts of $J_1$ and $J_1^+$ in $D$.
Since $H_1(X)$ is $\Z[\Z]$-torsion (see e.g.\ \cite{MR}*{\S2.3}), for every knot $\K\subset X$ there exists 
a $\lambda_{\K}\in\Z[Z]$ such that $\lambda_{\K}\K$ is null-homologous and hence bounds a $2$-chain $\zeta_K$.
For a link $(\K,\K')$ in $X$ the rational function 
$\lk_X(\K,\K')(t)=\lambda_{\K}^{-1}\sum_{n\in\Z} (\zeta_{\K}\cdot t^n\K')t^n$ is well-defined, i.e.\ does not
depend on the choices of $\lambda_{\K}$ and $\zeta_{\K}$ (see e.g.\ \cite{MR}*{\S2.3}).
Then by the Tsukamoto--Yasuhara theorem \cite{TY}*{Theorem 1.2}
$\nabla_L(z)=-z\nabla_K(z)\lk_X(\tilde J_1^+,\tilde J_1)(t)$, where $z=t^{1/2}-t^{-1/2}$.%
\footnote{The statement in \cite{TY} erroneously identifies $J_1$ with $K_1$ (which must be distinct)
and omits the sign --- which must be present, as is easily checked for the first two coefficients
(see Remark \ref{co2}).}

Let $J_1'$ be a close zero pushoff of $J_1$ (that is, one satisfying $\lk(J_1',J_1)=0$), disjoint from $J_1^+$,
and let $\tilde J_1'$ be its lift in $D$.
Since $J_1^+$ and $J_1'$ are both homotopic to $J_1$ in $S^3\but F$, they cobound a homotopy $H$ in $S^3\but F$.
On the other hand, $\lk(J_1^+,J_1)=\lk(J_1^+,K_1)=\lk(J_1^+,L)-\lk(J_1^+,K_2)=0-\lk(K_1,K_2)=-\lk(L)$.
Hence $(H\cdot J_1)=-\lk(L)$, where $H$ is oriented so that $\partial H=J_1^+-J_1'$.
Then $(\tilde H\cdot\tilde J_1)=-\lk(L)$, where $\tilde H$ is a lift of $H$ in $D$.
This easily implies (cf.\ \cite{TY}*{Lemma 2.1(a)}) that 
$\lk_X(\tilde J_1^+,\tilde J_1)(t)=\lk_X(\tilde J_1',\tilde J_1)(t)-\lk(L)$.
But $\lk_X(\tilde J_1',\tilde J_1)(t)$ is by definition the Kojima--Yamasaki $\eta$-function $\eta_\Lambda(t)$.
By Cochran's theorem \cite{Co}*{Theorem 7.1} $\eta_\Lambda(t)=C_\Lambda(x)$, where $x=(1-t)(1-t^{-1})$.
We have $x=2-t-t^{-1}=-z^2$.
Thus $\nabla_L(z)=-z\nabla_K(z)\big(C_\Lambda(-z^2)-\lk(L)\big)$.
\end{proof}

\begin{proof}[(b)] We may assume that $K_0$ and $J_0$ cobound an embedded annulus $A$, disjoint from $K_1$
and $b\big(I\x (0,1]\big)$.
Let $\bar b(I\x I)\subset A$ be a small band connecting an arc in $K_0$ with its pushoff in $J_0$.
Let us note that $J_0$ is a zero pushoff of $K_0$, that is, $\lk(J_0,K_0)=\lk(J_0,L)-\lk(J_0,K_1)=0-\lk(L)=0$.
Let $J_0'$ be another zero pushoff of $K_0$, disjoint from $A$ and $b(I\x I)$.

Now let $\bar L=(\bar K_0,\bar K_1)$, where $\bar K_0=K$ and $\bar K_1=-J_0$.
Let us note that $\lk(\bar L)=0$.
Let $\bar\Lambda=(\bar J_1,\bar K)$, where $\bar K=\bar K_0\#_{\bar b}\bar K_1$ and $\bar J_1:=-J_0'$
is a pushoff of $\bar K_1=-J_0$, disjoint from $\bar b(I\x I)$ and such that $\lk(\bar J_1,\bar L)=0$.

We have $\bar L=(K,-J_0)=\Lambda$ and $\bar K=K\#_{\bar b}(-J_0)=(K_0\#_b K_1)\#_{\bar b}(-J_0)$.
It is easy to see that there is an ambient isotopy along $A$ and then along $b(I\x I)$, taking $\bar K$ 
onto $K_1$ and keeping $\bar J_1=-J_0$ fixed.
Thus $\bar\Lambda$ is equivalent to $(-J_0,K_1)$, which in turn is equivalent to $(-K_0,K_1)$.
By Remark \ref{properties}(d) each $\beta_i(-K_0,K_1)=\beta_i(K_0,K_1)$.
Hence $-zC_L(-z^2)=-zC_{\bar\Lambda}(-z^2)=\nabla_{\bar L}(z)/\nabla_{\bar K}(z)=
\nabla_\Lambda(z)/\nabla_{K_1}(z)$.
\end{proof}

\begin{remark} One could try to extend $C_L$ to PL links with non-vanishing linking number by declaring 
$-z\hat C_L(-z^2):=\nabla_\Lambda(z)/\nabla_{K_1}(z)$.
However, this $\hat C_L$ generally depends on the choice of the band $b$.
Indeed, using the skein relation for the Conway polynomial it is easy to see that under a self-intersection
of $K_1$, which splits $K_1$ into two lobes $K_1^a$ and $K_1^b$, the coefficient $\hat\beta_1(L)$ of 
$\hat C_L(z)$ at $z^3$ jumps by $-\lk(K_1^b,K_0)^2$ or $-\lk(K_1^a,K_0)^2$ according as $b$ is attached to 
$K_1^a$ or $K_1^b$.
\end{remark}

\begin{remark} Theorem \ref{formula}(b) should be compared with the following results.

\begin{itemize}
\item If $K_1$ is unknotted, $\lk(L)=0$ and $K_0^s$ is obtained from $K_0$ by performing 
the $(+1)$-surgery on $K_1$,%
\footnote{That is, by making a $+1$ Dehn twist of $S^3\but K_1$ along a disk spanned by $K_1$.}
then $1+C_L(-z^2)=\nabla_{K_0^s}(z)/\nabla_{K_1}(z)$ \cite{JLWW}*{Corollary 5.6}
(see also \cite{PY}*{Theorem 5.1(5)}).

\item If $\lk(L)=0$, then $C_L(-z^2)$ is equivalent, up to units of $\Z[\Z]$, to 
$\Delta_{\tilde M_L}(y^2)/\nabla_{K_1}(z)$, where $z=y-y^{-1}$, $M_L$ is the $3$-manifold obtained by 
$0$-surgery on $L$, and $\Delta_{\tilde M_L}(t)$ is the Alexander polynomial of the infinite cyclic 
covering $\tilde M_L\to M_L$ associated to the homomorphism
$\pi_1(M_L)\to H_1(M_L)\simeq H_1(S^3\but L)\to H_1(S^3\but K_1)\simeq\Z$ \cite{KY}*{Theorem 1}.
\end{itemize}
\end{remark}

\begin{remark} \label{co2} For a two-component PL link $L=(K_0,K_1)$
\[\frac{\nabla_L(z)}{\nabla_{K_0}(z)\,\nabla_{K_1}(z)}=\alpha_0(L)z+\alpha_1(L)z^3+\alpha_2(L)z^5+\dots,\] where 
$\alpha_0(L)=\lk(L)$; if $\lk(L)=0$, then $\alpha_1(L)=\beta(L)$; and if $\beta(L)=\lk(L)=0$, then
$\alpha_2=\beta_2(L)+\delta(L)+\beta_2'(L)$ (see Remarks \ref{properties}(e) and \ref{mu} concerning $\beta_i'$
and $\delta$) \cite{Co2}.
\end{remark}

Theorem \ref{formula} yields

\begin{corollary} \label{rational} \cite{Co} If $L$ is a two-component PL link, then $C_L$ is rational.
\end{corollary}

\begin{remark} \label{conpol}
It should be mentioned that there is another good (perhaps, better) reason why $C_L$ is rational 
for PL links $L$.
The multi-variable Conway polynomial $\conpol_L(z_1,\dots,z_m)$ of an $m$-component PL link, or more generally 
of a link with a fixed partition $L=K_1\cup\dots\cup K_m$ into $m$ sublinks, is a polynomial with integer 
coefficients, unless $m=1$, in which case $\conpol_L(z)$ is the Laurent polynomial $z^{-1}\nabla_L(z)$ \cite{Me}.
It contains the same information as the Conway potential function $\Omega_L$ and satisfies Conway's first identity 
$\conpol_{L_+}-\conpol_{L_-}=z_i\conpol_{L_0}$ for self-intersections of $K_i$ (however, the relation between 
$\Omega_L$ and $\conpol_L$ is a lot more tricky than just a variable change) \cite{Me}.
The rational power series \[\conpol^*_L(z_1,\dots,z_m):=
\frac{\conpol_L(z_1,\dots,z_m)}{\nabla_{K_1}(z_1)\cdots\nabla_{K_m}(z_m)}\]
extends as a formal power series to wild links (``by continuity'') and the extension is an isotopy invariant 
(but of course need not be a rational power series) \cite{Me}.
A formula of J. T. Jin \cite{Jin1} relating the Kojima--Yamasaki $\eta$-function with the Conway potential 
function, when expressed in terms of Cochran's invariants and the multi-variable Conway polynomial, 
is as follows \cite{Me}: for a two-component PL link $L$ with $\lk(L)=0$,  
\[-C_L(-z^2)=z\frac{\partial\conpol_L^*(\zeta,z)}{\partial\zeta}\Big|_{\zeta=0}.\]
Since the right hand side is defined regardless of whether $\lk(L)=0$ and extends to wild links,
one can also define the left hand side for all $2$-component links by means of this formula.
The resulting extension $C_L^\conpol$ of $C_L$ coincides with $C_L^\mho$ when $\lk(L)=0$; in general
they are different (in particular, $C_L^\conpol$ is sensitive to changing the orientation of one of 
the components and $C_L^\mho$ is not) but contain the same information about the link, modulo the knowledge
of the linking number \cite{Me}.
\end{remark}

\begin{remark} For every pair $P$, $P'$ of rational power series J. T. Jin \cite{Jin2} constructed a PL link $L$ 
such that $C_L=P$ and $C'_L=P'$, where $C'_L(x)=\sum_{i=1}^\infty\beta'_i(L)x^i$ (see Remark \ref{properties}(e) 
concerning the $\beta'_i$).
It is also known that if $L=(K_0,K_1)$ is a PL link such that $K_0$ bounds a genus $g$ Seifert surface $\Sigma$ 
in $S^3\but K_1$ (that is, the closed surface $\Sigma\cup_\partial D^2$ is of genus $g$), then 
the $\beta_i(L)$ satisfy a linear recurrence relation of length $g+1$ \cite{GL}*{\S6}.
\end{remark}

\begin{example} \label{rational-example} Since there are only countably many rational power series 
(with integer coefficients), not all power series are rational. 
It is easy to construct explicit power series that are not rational.

(a) For instance, 
\begin{multline*}
(1-4x)^{-1/2}=\sum_{k=0}^\infty\frac{(-\frac12)(-\frac32)\cdots(-\frac12-k+1)}{k!}(-4x)^k=
\sum_{k=0}^\infty\frac{1\cdot 3\cdots(2k-1)}{k!}(2x)^k\\
=2\sum_{k=0}^\infty\frac{(2k-1)!}{(k-1)!k!}x^k
=2\sum_{k=0}^\infty\binom{2k-1}kx^k\in\Z[[x]].
\end{multline*}
If $(1-4x)^{-1/2}=P/Q$, where $P,Q\in\Z[x]$, then upon setting $x=1/8$ we get that $\sqrt 2$ is a rational 
number, which is a contradiction.

(b) Also, the power series $\sum_{k=0}^\infty k!x^k$ is not rational, since for a rational power series 
$P/Q=\sum_{i=0}^\infty n_ix^i$ the growth rate of the coefficients is bounded by $|n_k|\le Ck^dr^k$ 
for some $C,r>0$, where $d$ is the degree of $Q$ (see \cite{St}*{4.1.1(iii)}).
\end{example}

From Theorem \ref{main}, Corollary \ref{rational} and Examples \ref{w-example}, \ref{m-example} and 
\ref{rational-example}(b) we obtain

\begin{theorem}
The links $W_{1!,2!,3!,\dots}$ and $M_{1!,2!,3!,\dots}$ are not I-equivalent to any PL link.
\end{theorem}

\section{Acknowledgements}

I would like to thank L. Beklemishev for the reference to Higman's lemma.


\begin{thebibliography}{00}
\bib{AG}{article}{
   author={Ancel, F. D.},
   author={Guilbault, C. R.},
   title={Mapping swirls and pseudo-spines of compact $4$-manifolds},
   journal={Topology Appl.},
   volume={71},
   date={1996},
   pages={277--293},
   note={\href{https://www.sciencedirect.com/science/article/pii/0166864196000077}{Journal}}
}

\bib{Bi}{article}{
   author={Bing, R. H.},
   title={A simple closed curve that pierces no disk},
   journal={J. Math. Pures Appl.},
   volume={35},
   date={1956},
   pages={337--343}
}

\bib{Br}{article}{
   author={Brin, M.},
   title={Curves isotopic to tame curves},
   conference={
      title={Continua, Decompositions, Manifolds},
   },
   book={
      publisher={Univ. Texas Press},
      address={Austin, TX},
   },
   date={1983},
   pages={163--166}
}

\bib{Ca}{article}{
   author={Casson, A. J.},
   title={Link cobordism and Milnor's invariant},
   journal={Bull. London Math. Soc.},
   volume={7},
   date={1975},
   pages={39--40},
}

\bib{Co0}{article}{
   author={Cochran, T. D.},
   title={A topological proof of Stallings' theorem on lower central series of groups},
   journal={Math. Proc. Cambridge Philos. Soc.},
   volume={97},
   date={1985},
   pages={465--472},
}

\bib{Co}{article}{
   author={Cochran, T. D.},
   title={Geometric invariants of link cobordism},
   journal={Comment. Math. Helv.},
   volume={60},
   date={1985},
   pages={291--311},
}

\bib{Co2}{article}{
   author={Cochran, T. D.},
   title={Concordance invariance of coefficients of Conway's link polynomial},
   journal={Invent. Math.},
   volume={82},
   date={1985},
   pages={527--541}
}

\bib{Co3}{article}{
   author={Cochran, T. D.},
   title={Derivatives of links: Milnor's concordance invariants and Massey's
   products},
   journal={Mem. Amer. Math. Soc.},
   volume={84},
   date={1990},
   number={427},
}

\bib{Co4}{article}{
   author={Cochran, T. D.},
   title={$k$-cobordism for links in $S^3$},
   journal={Trans. Amer. Math. Soc.},
   volume={327},
   date={1991},
   pages={641--654},
   note={\href{https://www.ams.org/journals/tran/1991-327-02/S0002-9947-1991-1055569-2/}{Journal}}
}

\bib{Da}{book}{
   author={Daverman, R. J.},
   title={Decompositions of Manifolds},
   series={Pure Appl. Math.},
   volume={124},
   publisher={Academic Press},
   address={Orlando, FL},
   date={1986},
}

\bib{Fox}{article}{
   author={Fox, R. H.},
   title={A remarkable simple closed curve},
   journal={Ann. of Math.},
   volume={50},
   date={1949},
   pages={264--265},
}

\bib{Gif}{article}{
   author={C. H. Giffen},
   title={Disciplining dunce hats in 4-manifolds},
   date={1976},
   note={Unpublished manuscript}
}

\bib{Gi}{article}{
   author={Gillman, D. S.},
   title={Sequentially $1$-{\rm ULC} tori},
   journal={Trans. Amer. Math. Soc.},
   volume={111},
   date={1964},
   pages={449--456},
   note={\href{https://www.ams.org/journals/tran/1964-111-03/S0002-9947-1964-0162234-6/}{Journal}},
}

\bib{GL}{article}{
   author={Gilmer, P.},
   author={Livingston, C.},
   title={An algebraic link concordance group for $(p,2p-1)$-links in $S^{2p+1}$},
   journal={Proc. Edinburgh Math. Soc.},
   volume={34},
   date={1991},
   pages={455--462},
}

\bib{Hig}{article}{
   author={Higman, G.},
   title={Ordering by divisibility in abstract algebras},
   journal={Proc. London Math. Soc.},
   volume={2},
   date={1952},
   pages={326--336},
}

\bib{Hi}{book}{
   author={Hillman, J.},
   title={Algebraic Invariants of Links},
   series={Series on Knots and Everything},
   volume={32},
   publisher={World Sci.},
   address={River Edge, NJ},
   date={2002},
   note={Chapter 1 available at \href{http://www.maths.usyd.edu.au/u/jonh/ail12i.pdf}{author's webpage}}
}

\bib{Jin1}{article}{
   author={Jin, Gyo Taek},
   title={On Kojima's $\eta$-function of links},
   conference={
      title={Differential Topology},
      address={Siegen},
      date={1987},
   },
   book={
      series={Lecture Notes in Math.},
      volume={1350},
      publisher={Springer, Berlin},
   },
   date={1988},
   pages={14--30},
}

\bib{Jin2}{article}{
   author={Jin, Gyo Taek},
   title={The Cochran sequences of semi-boundary links},
   journal={Pacific J. Math.},
   volume={149},
   date={1991},
   pages={293--302},
}

\bib{JLWW}{article}{
   author={Jiang, B.},
   author={Lin, X.-S.},
   author={Wang, S.},
   author={Wu, Y.-Q.},
   title={Achirality of knots and links},
   journal={Topology Appl.},
   volume={119},
   date={2002},
   pages={185--208},
   note={\arxiv[GT]{9905158}},
}

\bib{KL}{article}{
   author={Kirk, P.},
   author={Livingston, C.},
   title={Vassiliev invariants of two component links and the Casson--Walker invariant},
   journal={Topology},
   volume={36},
   date={1997},
   pages={1333--1353},
}

\bib{KY}{article}{
   author={Kojima, S.},
   author={Yamasaki, M.},
   title={Some new invariants of links},
   journal={Invent. Math.},
   volume={54},
   date={1979},
   pages={213--228}
}

\bib{Kr}{article}{
   author={Krushkal, V. S.},
   title={Additivity properties of Milnor's $\bar\mu$-invariants},
   journal={J. Knot Theory Ram.},
   volume={7},
   date={1998},
   pages={625--637},
}

\bib{Li1}{article}{
   author={Li, Gui-Song},
   title={Covering Sato--Levine invariants},
   journal={Trans. Amer. Math. Soc.},
   volume={349},
   date={1997},
   pages={5031--5042},
   note={\href{https://www.ams.org/journals/tran/1997-349-12/S0002-9947-97-01885-0/}{Journal}}
}

\bib{Li}{article}{
   author={Li, Gui-Song},
   title={Kirk's invariant may have mixed terms},
   journal={Math. Proc. Cambridge Philos. Soc.},
   volume={125},
   date={1999},
   number={3},
   pages={425--431},
}

\bib{Mats}{article}{
   author={Matsumoto, Y.},
   title={Wild embeddings of piecewise linear manifolds in codimension two},
   conference={
      title={Geometric Topology (Proc. Georgia Topology Conf.)},
   },
   book={
      publisher={Academic Press},
      address={New York},
   },
   date={1979},
   pages={393--428},
}

\bib{Me}{article}{
author={Melikhov, S. A.},
title={Colored finite type invariants and a multi-variable analogue of the Conway polynomial},
note={\arxiv[GT]{0312007v2}}
}

\bib{M2}{article}{
   author    = {Melikhov, S. A.},
   title     = {Steenrod homotopy},
    journal  = {Russ. Math. Surv.},
       volume   = {64},
       year     = {2009},
       pages    = {469--551},
   note    = {\arxiv{0812.1407}},
   translation = {
    language = {Russian},
      journal   = {Uspekhi Mat. Nauk},
      volume    = {64:3},
      year      = {2009},
      pages     = {73--166},
   }
}

\bib{M3}{article}{
   author={Melikhov, S. A.},
   title={Gauss-type formulas for link map invariants},
   note={\arxiv{1711.03530}},
}

\bib{MR}{article}{
   author={Melikhov, S. A.},
   author={Repov\v{s}, D.},
   title={$n$-Quasi-isotopy. I. Questions of nilpotence},
   journal={J. Knot Theory Ramifications},
   volume={14},
   date={2005},
   pages={571--602},
   note={\arxiv[GT]{0103113}}
}

\bib{Mi}{article}{
   author={Milnor, J.},
   title={Isotopy of links},
   conference={
      title={Algebraic Geometry and Topology: A Symposium in Honor of S. Lefschetz},
   },
   book={
      publisher={Princeton Univ. Press},
      address={Princeton, N. J.},
   },
   date={1957},
   pages={280--306},
}

\bib{PY}{article}{
   author={Przytycki, J. H.},
   author={Yasuhara, A.},
   title={Linking numbers in rational homology 3-spheres, cyclic branched covers and infinite cyclic covers},
   journal={Trans. Amer. Math. Soc.},
   volume={356},
   date={2004},
   number={9},
   pages={3669--3685},
}

\bib{Ro}{article}{
   author={Rolfsen, D.},
   title={Some counterexamples in link theory},
   journal={Canadian J. Math.},
   volume={26},
   date={1974},
   pages={978--984},
}

\bib{Orr}{article}{
   author={Orr, K. E.},
   title={Homotopy invariants of links},
   journal={Invent. Math.},
   volume={95},
   date={1989},
   number={2},
   pages={379--394},
}

\bib{Sta}{article}{
   author={Stallings, J.},
   title={Homology and central series of groups},
   journal={J. Algebra},
   volume={2},
   date={1965},
   pages={170--181},
}

\bib{St}{book}{
   author={Stanley, R. P.},
   title={Enumerative Combinatorics, Vol. 1},
   series={Cambridge Studies in Adv. Math.},
   volume={49},
   publisher={Cambridge Univ. Press, Cambridge},
   edition={Corrected reprint},
   date={1997},
}

\bib{Ste}{article}{
   author={Stein, D.},
   title={Massey products in the cohomology of groups with applications to link theory},
   journal={Trans. Amer. Math. Soc.},
   volume={318},
   date={1990},
   pages={301--325},
   note={\href{https://www.ams.org/journals/tran/1990-318-01/S0002-9947-1990-0958903-3/}{Journal}}
}

\bib{SB}{article}{
   author={Sturm Beiss, R.},
   title={The Arf and Sato link concordance invariants},
   journal={Trans. Amer. Math. Soc.},
   volume={322},
   date={1990},
   pages={479--491},
   note={\href{https://www.ams.org/journals/tran/1990-322-02/S0002-9947-1990-1012525-7/}{Journal}}
}

\bib{TY}{article}{
   author={Tsukamoto, T.},
   author={Yasuhara, A.},
   title={A factorization of the Conway polynomial and covering linkage invariants},
   journal={J. Knot Theory Ram.},
   volume={16},
   date={2007},
   pages={631--640},
   note={\arxiv[GT]{0405481}},
}


\end{thebibliography}
\end{document}